\documentclass{article}
\usepackage{amssymb}
\usepackage{amscd}
\usepackage{amsmath,amssymb}
\usepackage{amsthm}
\usepackage{amsfonts}
\usepackage{latexsym}
\author{Milan Janji\'c}
\title{Determinants and Compositions of Natural Numbers}
\begin{document}
\maketitle

\begin{abstract}
We consider a particular type of matrices which belong at the same time to the class of Hessenberg and Toeplitz matrices, and  whose determinants are equal
to the number of a type of  compositions of natural numbers.

We prove a formula in which the number of weak compositions with a fixed number of zeroes is expressed in terms of the number of compositions without zeroes. Then we find a relationship between weak compositions and coefficients of characteristic polynomials of appropriate matrices.

 Finally, we prove three explicit formulas for weak compositions of a special kind.
 \end{abstract}
\newtheorem{theorem}{Theorem}
\newtheorem{corollary}[theorem]{Corollary}
\newtheorem{lemma}[theorem]{Lemma}
\newtheorem{proposition}[theorem]{Proposition}
\newtheorem{conjecture}[theorem]{Conjecture}
\newtheorem{defin}[theorem]{Definition}
\newenvironment{definition}{\begin{defin}\normalfont\quad}{\end{defin}}
\newtheorem{examp}[theorem]{Example}
\newenvironment{example}{\begin{examp}\normalfont\quad}{\end{examp}}
\newtheorem{rema}[theorem]{Remark}
\newenvironment{remark}{\begin{rema}\normalfont\quad}{\end{rema}}

\section{Introduction}
We investigate different type of compositions of natural numbers. The following notation will be used. Let $n$ be a positive integer, and $1\leq m_1<m_2<\cdots<m_r\leq n$ be arbitrary integers. We let $M$ denote the set consisting of $q_1$ different types of $m_1,$ $q_2$ different type of $m_2,$ and so on. We denote $c(n,M')$ the number of compositions of $n$ whose each part belongs to $M.$ Also,
$cw(n,k,M')$ denote the number of weak compositions of the same kind in which there are exactly $k$ zeroes. The standard compositions are obtained in the case that all $q_i$ are equal $1.$ In this case we shall write $c(n,M),$ for standard compositions,  and $cw(n,k,M)$ for weak compositions.

As usual, we denote by $[n]$  the set $\{1,2,\ldots,n\}.$ We also write $c(n),$ and $cw(n,k)$ instead of $c(n,[n]),$ and $cw(n,k,[n]).$
Fibonacci numbers will be denoted by $F_1,F_2,\ldots,$ with $F_1=F_2=1.$

In the second section of the paper we first prove a theorem connecting determinants of some matrices with recurrence sequences.  We then prove that coefficients of characteristic polynomials of such matrices, that is, its sums of principal minors,  are convolutions of terms of appropriate  sequences.

In the third section we consider recurrence sequences whose terms represent the numbers of compositions.
We derive formulas in which the numbers of compositions are equal to some determinants. Then we prove a formula in which the weak compositions are expressed in term of compositions without zeroes. From this we derive a formula in which the weak compositions are obtained as the sums of principal minors of appropriate matrices.

In the fourth section we prove three explicit formulas for the weak compositions. The first formula concerns unrestricted compositions, the second formula deals with compositions in which positive parts are either $1$ or $2.$ Finally, the third formula concerns compositions in which all parts are $\geq 2.$

We shall need the following result, proved in \cite{mil} (Proposition 2.3).

\begin{equation}\label{e1}\sum_{j_1+j_2+\cdots+j_{k+1}=n-k}F_{j_1+1}F_{j_2+1}\cdots F_{j_{k+1}+1}=\sum_{i=0}^{\lfloor\frac{n-k}{2}\rfloor}
{n-i\choose i}{n-2i\choose k},\end{equation}
where the sum is taken over $j_t\geq 0.$
Note that in Riordan's book \cite{riordan} the sum on the left side is called the convolved Fibonacci number.

\section{Determinants and Recurrence Sequences}
The following results about upper Hessenberg ( Toeplitz)  matrices will be used in the paper.

\begin{theorem}\label{t1} Let the matrix $P_n$ be defined by:
\begin{equation}\label{mat}P_n= \left[\begin{array}{rrrrrr} p_{1}&p_{2}&p_{3}&\cdots&p_{n-1}&p_{n}\\
-1&p_{1}&p_{2}&\cdots&p_{n-2}&p_{n-1}\\ 0&-1&p_{1}&\cdots&p_{n-3}&p_{n-2}\\ \vdots&\vdots&\vdots&\ddots&\vdots&\vdots\\
0&0&0&\cdots&p_{1}&p_{2}\\ 0&0&0&\cdots&-1&p_{1}\end{array}\right],\end{equation}
 and let the sequence $a_1,a_2,a_3,\ldots,\;(a_1\not=0)$ be defined by:
\begin{equation}\label{niz}\;a_{n+1}=\sum_{i=1}^{n}p_{n-i+1}a_i,\;(n=1,2,\ldots).\end{equation}
Then \[a_{n+1}=a_1\det
P_n,\;(n=1,2,\ldots).\]
\end{theorem}
\begin{proof} We use induction of with respect to $n.$ For $n=1$ the theorem obviously holds. Assume the theorem is true for $k<n.$
  Expanding $\det P_n$ across the last column we obtain
   \[\det P_n=\sum_{i=1}^n(-1)^{i+n}p_{n-i+1}M_{in},\] where $M_{in}$ is the minor of  order $n-1$ obtained by deleting the $i$th row and the $n$th column of $P_n.$
 The minor $M_{in}$ is the determinant of an upper triangular block matrix in which the upper left block is $P_{i-1},$ and the lower right block is an upper triangular matrix of order $n-i$ with $-1$ on the main diagonal. It follows that
\[ M_{in}=(-1)^{n-i}P_{i-1}.\] We conclude that
\[a_1\det P_n=a_1\sum_{i=1}^np_{n-i+1}P_{i-1}.\] Using the induction hypothesis we obtain
\[\det P_n=\sum_{i=1}^np_{n-i+1}a_{i}=a_{n+1},\] and the theorem is proved.
\end{proof}
Throughout the paper $a_1$ will have the value $1.$

We shell next prove that sums of principal minors of a fixed order of $P_n$ are convolutions of the terms of the sequence (\ref{niz}).
\begin{theorem}Let $S_{n-k}$ be the sum of all principal minors of order $n-k$ of the matrix $P_n.$ Then
\begin{equation}\label{snkk}S_{n-k}=\sum_{j_1+j_2+\cdots+j_{k+1}=n-k}a_{j_1+1}a_{j_2+1}\cdots a_{j_{k+1}+1},\end{equation}
where the sum is taken over  $j_t\geq
 0,\;(t=1,2,\ldots,k+1).$
\end{theorem}
\begin{proof}
The matrix obtained by deleting the $i$th row and
 column of $P_n$ is  an upper triangular matrix which the upper left block is $P_{i-1},$ and
the lower right block is  $P_{n-i}.$
Its determinant, by (\ref{niz}), is $a_{i}\cdot a_{n-i+1}.$ Therefore, the principal minor $M(i_1,i_2,\ldots,i_k)$ of $P_n$
obtained by deleting the rows and columns with indices $1\leq i_1<i_2<\cdots<i_k\leq n$ is \[M(i_1,i_2,\ldots,i_k)=a_{i_1}\cdot
a_{i_2-i_1}\cdots a_{i_k-i_{k-1}}\cdot a_{n-i_k+1}.\]

Denoting $i_1=j_1+1,i_2-i_1=j_2+1,\ldots, i_k-i_{k-1}=j_k+1,n-i_k+1=j_{k+1}+1$ yields
\[M(i_1,i_2,\ldots,i_k)=a_{j_1+1}\cdot
a_{j_2+1}\cdots a_{j_k+1}\cdot a_{j_{k+1}+1},\]
where $j_t\geq 0,\;(t=1,2,\ldots,k+1),$ and $j_1+\cdots+j_{k+1}=n-k.$

 Summing over all  $1\leq i_1<i_2<\cdots<i_k\leq n$ we obtain (\ref{snkk}).
\end{proof}

  \section{Restricted Compositions}
We first prove that, in fact, terms of each recurrence sequence of the form (\ref{niz}), provided that all $p_i$ are positive, count a type of compositions. Namely, we have
\begin{theorem} The following formula is true
\begin{equation}\label{t2}c(n,M')=q_kc(n-m_k,M')+q_{k-1}c(n-m_{k-1},M')+\cdots+q_1c(n-m_1,M').\end{equation}
\end{theorem}
\begin{proof} The first term on the right side is the number of compositions of $n$ ending with one of $m_k,$ and so on.
\end{proof}

The matrix $P_{n}=(p_{ij})$  corresponding to this recurrence in sense of Theorem \ref{t1} is defined as follows
\[\begin{array}{ccc}
p_{i,j}=-1,&\mbox{ if }&i=j+1,\\
p_{i,j}=q_k,&\mbox{ if }&j-i=m_k-1,\\
p_{i,j}=0,&& \mbox{ otherwise}.
\end{array}\]

\noindent
As an immediate consequence of Theorem \ref{t1} we obtain
\begin{theorem}\label{t3} The following formula is true \[c(n,M')=\det P_n.\]
\end{theorem}

We shall state some particular case of Theorem \ref{t3}.
\begin{corollary} The number $c(n)$ of all compositions of $n$ is $2^{n-1}.$
\end{corollary}
\begin{proof}
Theorem \ref{t2} yields
 then \[c(n)=c(n-1)+\cdots+1+1.\]
It is easily seen that $c(n)=2^{n-1}.$
\end{proof}
\begin{corollary}
Let $F_n^{(k)}$ be $k$-step Fibonacci number. Then $F_{n+1}^{(k)}$ counts the compositions of $n$ in which all parts are $\leq k.$
\end{corollary}
\begin{proof} We have $M=[k],\;(k<n),$ and equation (\ref{niz}) implies
\[c(n,M)=c(n-1,M)+c(n-2,P)+\cdots+c(n-k,M).\] Takin into account that $c(0,P)=1=F_1^{(k)},\;c(1,P)=1=F_2^{(k)}$ we see that $c(n,P)=F^{(k)}_{n+1}.$
\end{proof}
\begin{rema} It is known that numbers $_kF_n$ defined in \cite{and} count the number of compositions of $n+1$ in which each part is $\geq k.$
The matrix $P_n$ for such compositions is defined in the following way:
\[\begin{array}{ccc}
p_{i,j}=-1,&\mbox{ if }&i=j+1,\\
p_{i,j}=1,&\mbox{ if }&j-i\geq k,\\
p_{i,j}=0,&& \mbox{ otherwise}.
\end{array}\]

\end{rema}

Denote by $cc(n)$ the number of compositions of $n$ of any type, and   denote by $ccw(n,k)$ the number of such weak compositions in which there are exactly $k$ zeroes. The next result is a formula expressing $ccw(n,k)$ in terms of $cc(j),\;(j\leq n).$
\begin{theorem} The following formula is true
\begin{equation}\label{e2}ccw(n,k)=\sum_{j_1+j_2+\cdots+j_{k+1}=n}cc(j_1)cc(j_2)\cdots cc(j_{k+1}),\end{equation}
where the sum is taken over  $j_t\geq 0,\;(t=1,2,\ldots,k+1),$
with $cc(0)=1.$
\end{theorem}
\begin{proof} We use induction with respect to $k.$ For $k=0$ the assertion is obvious. Assume that the assertion is true for $k-1.$
We may write  equation (\ref{e2}) in the following form:
\begin{equation}\label{r1}ccw(n,k)=\sum_{j=0}^ncc(j)ccw(n-j,k-1).\end{equation}
Let $(i_1,i_2,\ldots,)$ be a weak composition of $n$ in which
exactly $k$ parts equal $0.$  Assume that $i_p$ is the first part equals zero.
Then $(i_1,\ldots,i_{p-1})$ is a composition of
$i_1+\cdots+i_{p-1}=j$ without zeroes, and $(i_{p+1},\cdots)$ is a
weak compositions of $n-j$ with $k-1$ zeroes. For a fixed $j$
there are $cc(j)ccw(n-j,k-1)$ such compositions. Changing $j$ we
conclude that the right side of (\ref{r1}) counts all weak
compositions. Applying the induction hypothesis we prove the theorem.
\end{proof}

Define the sequence $a_n,\;(n=1,2,\ldots)$ in the following way: \[a_1=1,a_2=cc(1),\ldots,a_n=cc(n-1),\ldots,\] and assume that (\ref{t1}) is the recurrence relation for compositions. Theorem \ref{t2} implies that $c(n)=\det P_n.$
In this case equation (\ref{snkk}) has the form
\[S_{n-k}=\sum_{j_1+j_2+\cdots+j_{k+1}=n-k}cc(j_1)cc(j_2)\cdots cc(j_{k+1}).\]
Comparing this equation with (\ref{e2}) we obtain
\begin{theorem} Let  $n$ be a positive integer. Then
\[cw(n,k,M')=S_n,\]
where $S_n$ is the sum of all principal minors of order $n$ of the matrix $P_{n+k}.$
\end{theorem}

\section{Three Explicit Formulas}

In this section we shall derive a few explicit formulas for weak
composition.
\begin{theorem}
Let $n$ be a positive integer, and let $k$ be a nonnegative integer. Then
\[cw(n,k)=2^{n-k-1}\sum_{i=0}^k2^i{k+1\choose i}{n-1\choose k-i}.\]
\end{theorem}
\begin{proof}
In the case we have $c(0)=1,\;c(j)=2^{j-1},\;(j\geq 1).$
Collecting  terms in which some $j_t$ are zeroes in (\ref{e2}) we
obtain
\[c(n,k)=\sum_{m=0}^k{k+1\choose m}\sum_{j_1+\ldots+j_{k+1-m}=n}c(j_1)c(j_2)\cdots c(j_{k+1-m}),\]
where the second  sum is taken over $j_t\geq 1.$ It follows that
\[\sum_{j_1+j_2+\cdots+j_{k+1-m}=n}c(j_1)c(j_2)\cdots c(j_{k+1-i})=2^{n-k+m-1}\sum_{j_1+j_2+\cdots+j_{k+1-m}=n}1.\]
The theorem holds since the last sum is the number of composition
of $n$ into $k+1-m$ parts.
\end{proof}

\begin{theorem} The following formula is true
\[cw(n,k,[2])=\sum_{i=0}^{\lfloor\frac{n}{2}\rfloor}
{n+k-i\choose i}{n+k-2i\choose k}.\]
\end{theorem}
\begin{proof}
Let $a_1,a_2,\ldots$ be the sequence defined by
$a_1=1,a_i=F_{i+1},(i>1).$ It is well-known that
$a_i=c(n,[2]),\;(i>1).$ In this case $P_{n}$ is well-known matrix
which determinants is $F_{n+1}.$ From (\ref{snkk}) we have
\[S_{(n+k)-k}=S_{n}=\sum_{j_1+j_2+\cdots+j_{k+1}=n}F_{j_1+1}F_{j_2+1}\cdots
F_{j_{k+1}+1},\] where the sum is taken over $j_t\geq 0.$
Hence, the theorem follows from (\ref{e1}).
\end{proof}
\begin{theorem} If $M=\{2,3,\ldots,n\}$ then
\[cw(n+k-1,k,M)=\sum_{m=0}^{k+1}\sum_{i=0}^{\lfloor\frac{n-k-1+m}{2}\rfloor}
{k+1\choose m}\cdot X,\] where \[X={n-1-i\choose i}{n-1-2i\choose k-m}
.\]
\end{theorem}
\begin{proof}
We now that $F_{n-1}$ is the number of compositions of $n$ in which all parts are $\geq 2,$
Taking $a_1=1,\;a_n=F_{n-1},(n\geq 2)$ we have $a_{n+1}=a_{n-1}+a_{n-2}+\cdots a_1.$
Let  $P_n$ be the matrix which suits to this recurrence relation.
According (\ref{snkk}) we have

\[S_n=S_{(n+k)-k}=\sum_{m=0}^k{k+1\choose m}\sum_{j_1+\ldots+j_{k+1-m}=n}F_{j_1}F_{j_2}\cdots F_{j_{k+1-m}},\]
where the second  sum is taken over $j_t\geq 1.$
Replacing $j_t,\;(t=1,2,\ldots,k+1-m)$ with $i_t+1$ we obtain
\[S_{n}=\sum_{m=0}^k{k+1\choose m}\sum_{i_1+\ldots+i_{k+1-m}=n-k-1+m}F_{i_1+1}F_{i_2+1}\cdots F_{i_{k+1-m}+1},\]
where the sum is taken over $i_t\geq 0.$
The theorem is true according to (\ref{e1}).

\end{proof}

\end{document}